\documentclass{amsart}

\RequirePackage{color}

\usepackage[utf8]{inputenc}
\def\smallskip{\vskip\smallskipamount}
\def\medskip{\vskip\medskipamount}
\def\bigskip{\vskip\bigskipamount}

\usepackage{amsmath,amsthm,bm, mathrsfs, amscd, tikz, csquotes}
\usepackage[maxbibnames=99]{biblatex}
\usepackage{ytableau}
\usepackage{comment}
\usepackage{mathdots}
\usepackage[all]{xy}
\usepackage{graphicx}
\usepackage[colorinlistoftodos]{todonotes}
\usepackage{enumerate}
\usepackage{feynmp-auto}
\usepackage[english]{babel}
\usepackage{float}
\usepackage{tikz}
\usepackage{tikz-cd}
\usepackage{amssymb, hyperref, csquotes}
\usepackage{mathtools}
\usetikzlibrary{decorations.pathmorphing}
\usepackage{subcaption}

\usepackage[utf8]{inputenc}

\newtheoremstyle{thmstyle}{}{}{\itshape}{}{\bfseries}{ }{5pt}{}
\newtheoremstyle{exstyle}{}{}{}{}{\bfseries}{ }{5pt}{}
\newtheoremstyle{defstyle}{}{}{}{}{\bfseries}{ }{5pt}{}
\newtheoremstyle{remstyle}{}{}{}{}{\bfseries}{ }{5pt}{}

\theoremstyle{thmstyle}
\newtheorem{thm}{Theorem}[section]
\newtheorem{theorem}[thm]{Theorem}

\newtheorem{lemma}[thm]{Lemma}

\newtheorem{proposition}[thm]{Proposition}

\newtheorem{corollary}[thm]{Corollary}

\theoremstyle{exstyle}

\newtheorem{example}[thm]{Example}

\theoremstyle{defstyle}

\newtheorem{definition}[thm]{Definition}

\newtheorem{def-prop}[thm]{Definition-Proposition}
\newtheorem{def-lem}[thm]{Definition-Lemma}
\newtheorem{rem-convention}[thm]{Remark-Convention}
\newtheorem{def-note}[thm]{Definition-Notation}

\theoremstyle{remstyle}

\newtheorem{remark}[thm]{Remark}

\theoremstyle{remstyle}

\newcommand{\ZZ}{\mathbb Z}

\newcommand{\D}{\mathcal{D}}

\newcommand{\gc}{{\color{teal} \circ}}
\newcommand{\rc}{{\color{red} \bullet}}

\newcommand{\La}{\Lambda}

\newcommand{\proj}{\operatorname{proj}}
\newcommand{\mmod}{\operatorname{mod}}
\newcommand{\K}{\mathrm{K}}
\newcommand{\Ext}{\mathrm{Ext}}
\newcommand{\Hom}{\mathrm{Hom}}

\newcommand{\Address}{{
  \bigskip
  \footnotesize

  E.~Gupta, \textsc{Université Paris-Saclay, UVSQ, CNRS, Laboratoire de Mathématiques de Versailles, 78000, Versailles, France.}\par\nopagebreak
  \textit{E-mail address}: \texttt{esha.gupta2@uvsq.fr}
}}

\makeatletter
\newcommand{\doublewidetilde}[1]{{%
  \mathpalette\double@widetilde{#1}%
}}
\newcommand{\double@widetilde}[2]{%
  \sbox\z@{$\m@th#1\widetilde{#2}$}%
  \ht\z@=.9\ht\z@
  \widetilde{\box\z@}%
}
\makeatother

\title[A restricted model for derived categories of gentle algebras]{A restricted model for the bounded derived category of gentle algebras}
\author{Esha Gupta}
\date{}

\begin{document}
\begin{abstract}
    We present a restricted model for the bounded derived category of gentle algebras that encodes the indecomposable objects and positive extensions between them. The model is then used to count the number of $d$-term silting objects for linearly oriented $A_n$, recovering the result that they are counted by the Pfaff-Fuss-Catalan numbers. 
\end{abstract}
\maketitle
\tableofcontents{}
\section{Introduction}
In recent years, (graded) gentle algebras have come to gather a lot of attention because of their ubiquity in several areas of mathematics. These include cluster theory, where they appear as cluster-tilted algebras and Jacobian algebras associated to triangulations of marked surfaces \cite{ABCP, L}, as well as geometry, where they appear as endomorphism rings of formal generators of some partially wrapped Fukaya categories \cite{HKK}. A geometric model for the module categories of all gentle algebras was provided in \cite{BC}, where the authors realised them as tiling algebras associated to partial triangulations of unpunctured surfaces with marked points on the boundary. A geometric model for the $m$-cluster category of type $A_{n-1}$ was provided in \cite{BM} using diagonals of a regular polygon. In \cite{OPS}, the authors provided a complete model for the bounded derived category of a gentle algebra, which encoded information such as the indecomposable objects, morphisms, mapping cones, Auslander-Reiten translation for perfect objects, and Auslander–Reiten triangles. A model for studying support $\tau$-tilting modules, or equivalently $2$-term silting complexes, of all locally gentle algebras was provided in \cite{PPP19}. For an acyclic quiver $Q$, it was shown in \cite[Theorem~5.14]{KQ} that the hearts of $\D^b(\mmod kQ)$ lying between the canonical heart $\mathcal{H}_Q=\mmod kQ$ and $\mathcal{H}_Q[d-1]$ are in bijection with basic $d$-term silting complexes as well as $d$-cluster tilting sets.

In this work, we modify the model of \cite{OPS} to give a model for the bounded derived category of a gentle algebra that encodes the indecomposable objects and their positive extensions as (ungraded) arcs on a surface and their intersections (Proposition \ref{posext}). We can then restrict this model to study the truncated homotopy categories $\K^{[-d+1,0]}(\proj\La)$ of $d$-term complexes of projectives for a gentle algebra $\La$. Since the model encodes positive extensions, it is particularly suited for the study of $d$-term silting complexes generalizing the case of $2$-term silting complexes studied in \cite{PPP19} (Corollary \ref{dsilt}, Remark \ref{d=2}). Applying the model to the case of linearly oriented $A_n$ quiver, we recover the result of \cite{STW} that the number of $d$-term silting complexes in $\K^b(\proj kA_n)$ is given by the Pfaff-Fuss-Catalan number $C^d_{n+1}$ (\S~\ref{count}).

\section{Marked surfaces and admissible dissections}
Throughout this work, $k$ will denote an algebraically closed field. We recall here the special definition of marked surfaces as used in \cite{APS}, the admissible dissections of which are in bijection with gentle algebras.  
\begin{definition}\cite[Definition~1.7]{APS}
A \emph{marked surface} is a triple $(S, M, P)$, where
\begin{enumerate}
    \item $S$ is an oriented closed smooth surface with non-empty boundary $\partial S$;
    \item $M = M_{\color{teal}\circ} \cup M_{\color{red}\bullet}$ is a finite set of marked points on $\partial S$. The elements of $M_{\color{teal}\circ}$ and $M_{\color{red}\bullet}$ will be represented by symbols $\color{teal}\circ$ and $\color{red}\bullet$, respectively. They are required to alternate on each connected component of $\partial S$, and each such component is required to contain at least one marked point;
    \item $P = P_{\color{red}\bullet}$ is a finite set of marked points in the interior of $S$, called punctures. The elements of $P_{\color{red}\bullet}$ will also be represented by $\color{red}\bullet$.
\end{enumerate}
\end{definition}
\begin{definition}\cite[Definition~1.8]{APS}
A $\color{teal}\circ$\emph{-arc} (or $\color{red}\bullet$\emph{-arc}) is a smooth map $\gamma$ from~$[0, 1]$ to $S \setminus P$ such that
its endpoints $\gamma(0)$ and $\gamma(1)$ are in $M_{\color{teal}\circ}$ (or in $M_{\color{red}\bullet}$, respectively). The curve $\gamma$ is required not to be contractible to a point in $M_{\color{teal}\circ}$ (or $M_{\color{red}\bullet}$, respectively).    
\end{definition} 
We will usually consider arcs up to homotopy and inverses. Two arcs are said to \emph{intersect} if any choice of homotopic representatives intersect.

\begin{definition}\cite[Definition~1.9]{APS}
A collection of pairwise non-intersecting and pairwise different $\color{teal}\circ$-arcs $\{\gamma_1, \cdots , \gamma_r\}$
on the surface $(S, M, P)$ is called \emph{admissible} if the arcs $\{\gamma_1, \cdots , \gamma_r\}$ do not enclose a subsurface containing no punctures of $P_{\color{red}\bullet}$ and with no boundary segment on its boundary. A maximal admissible collection
of $\color{teal}\circ$-arcs is called an \emph{admissible $\color{teal}\circ$-dissection}.    
\end{definition}
The notion of admissible $\color{red}\bullet$-dissection is defined similarly.

To any admissible $\color{teal}\circ$-dissection, we can associate a dual $\color{red}\bullet$-dissection in the following sense.

\begin{proposition}\cite[Proposition~1.13]{APS}
Let $(S, M, P)$ be a marked surface, and let $\Delta$ be an admissible $\color{teal}\circ$-dissection. There exists a unique admissible $\color{red}\bullet$-dissection~$\Delta^*$ (up to homotopy) such that each arc of $\Delta^*$ intersects exactly one arc of $\Delta$.    
\end{proposition}

\begin{definition}\cite[Definition~1.14]{APS}
The dissections $\Delta$ and $\Delta^*$ are called \emph{dual dissections}.    
\end{definition}

\section{Admissible dissections and gentle algebras}
\begin{definition}\cite[Definition~2.1]{APS}
Let $\Delta$ be an admissible $\color{teal}\circ$-dissection of a marked surface $(S, M, P)$. The $k$-algebra $A(\Delta)$ is the quotient of the path algebra of the quiver $Q(\Delta)$ by the ideal $I(\Delta)$ defined as follows:
\begin{itemize}
    \item the vertices of $Q(\Delta)$ are in bijection with the $\color{teal}\circ$-arcs in $\Delta$;
    \item there is an arrow $i \to j$ in $Q(\Delta)$ whenever the $\color{teal}\circ$-arcs $i$ and $j$ meet at a marked point $\color{teal}\circ$, with $i$ preceding $j$ in the counter-clockwise order around~$\color{teal}\circ$, and with no other arc coming to $\color{teal}\circ$ between $i$ and $j$. 
    \item the ideal $I(\Delta)$ is generated by the following relations: whenever $i$ and $j$ meet at a marked point as above, and the other end of $j$ meets $k$ at a marked point as above, then the composition of the corresponding arrows~$i \to j$ and $j \to k$ is a relation.   
\end{itemize}  
\end{definition} 
\begin{example}
    Let $(S,M,P)$ be the marked surface in Figure \ref{marksurf}, and $\Delta$ the $\color{teal}\circ$-dissection dual to the depicted $\color{red}\bullet$-dissection. Then $A(\Delta)$ is the path algebra of the quiver with relations shown in Figure \ref{gen_alg}.
    \begin{figure}[H]
        \centering
        \begin{subfigure}{.5\textwidth}
        \centering
        \includegraphics[scale=0.3]{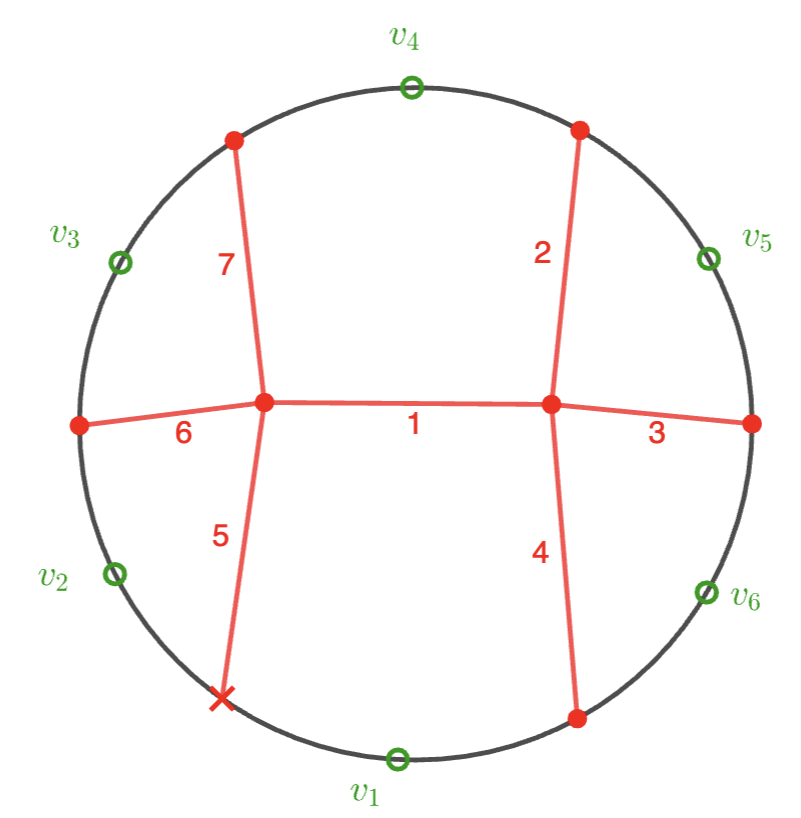}
        \caption{}
        \label{marksurf}
        \end{subfigure}%
        \begin{subfigure}{.5\textwidth}
        \centering
        \includegraphics[scale=0.4]{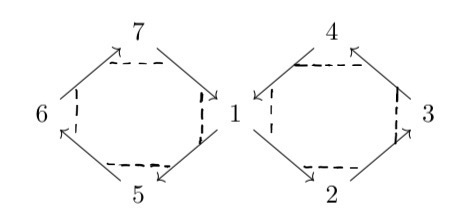}
        \caption{}
        \label{gen_alg}
        \end{subfigure}
        \caption{}
    \end{figure}

\end{example}
\begin{theorem}\cite[Theorem~2.3]{APS}\label{MStoGA}
The assignment
$((S, M, P), \Delta)\to
A(\Delta)$ defines a bijection from the set of homeomorphism classes of marked surfaces $(S, M, P)$ with an admissible dissection to the
set of isomorphism classes of gentle algebras.
\end{theorem}

Let $\Lambda$ be a gentle algebra of rank $n$. Let $((S,M,P),\Delta)$ be the marked surface with an admissible dissection obtained from the previous theorem. We will give another model of the bounded derived category of $\Lambda$ using this surface. 

We start by choosing a labelling of the $\gc$-points in $M$. We replace each $\gc$-point $v$ on a boundary component between two $\rc$-points with a collection of $\mathbb{Z}$-indexed $\color{blue}\times$-points $T=\{T_{(i,v)}\mid i\in\ZZ\}$ arranged in descending order along the orientation of the component. Moreover, we do so in such a way that the set~$T$ has precisely two limit points given by the two $\rc$-points. For a point $T_{(i,v)}$, we define~$\epsilon(T_{(i,v)}):=i$ and $\sigma(T_{(i,v)}):=v$. We consider special arcs joining these~$\color{blue}\times$-points which we call `slaloms'.

\begin{example}
The above process transforms the marked surface in Figure \ref{marksurf2} to the surface in Figure \ref{transform}.
    \begin{figure}[H]
        \centering
        \begin{subfigure}{.5\textwidth}
        \centering
        \includegraphics[scale=0.3]{marked_surface.png}
        \caption{}
        \label{marksurf2}
        \end{subfigure}%
        \begin{subfigure}{.5\textwidth}
        \centering
        \includegraphics[scale=0.5]{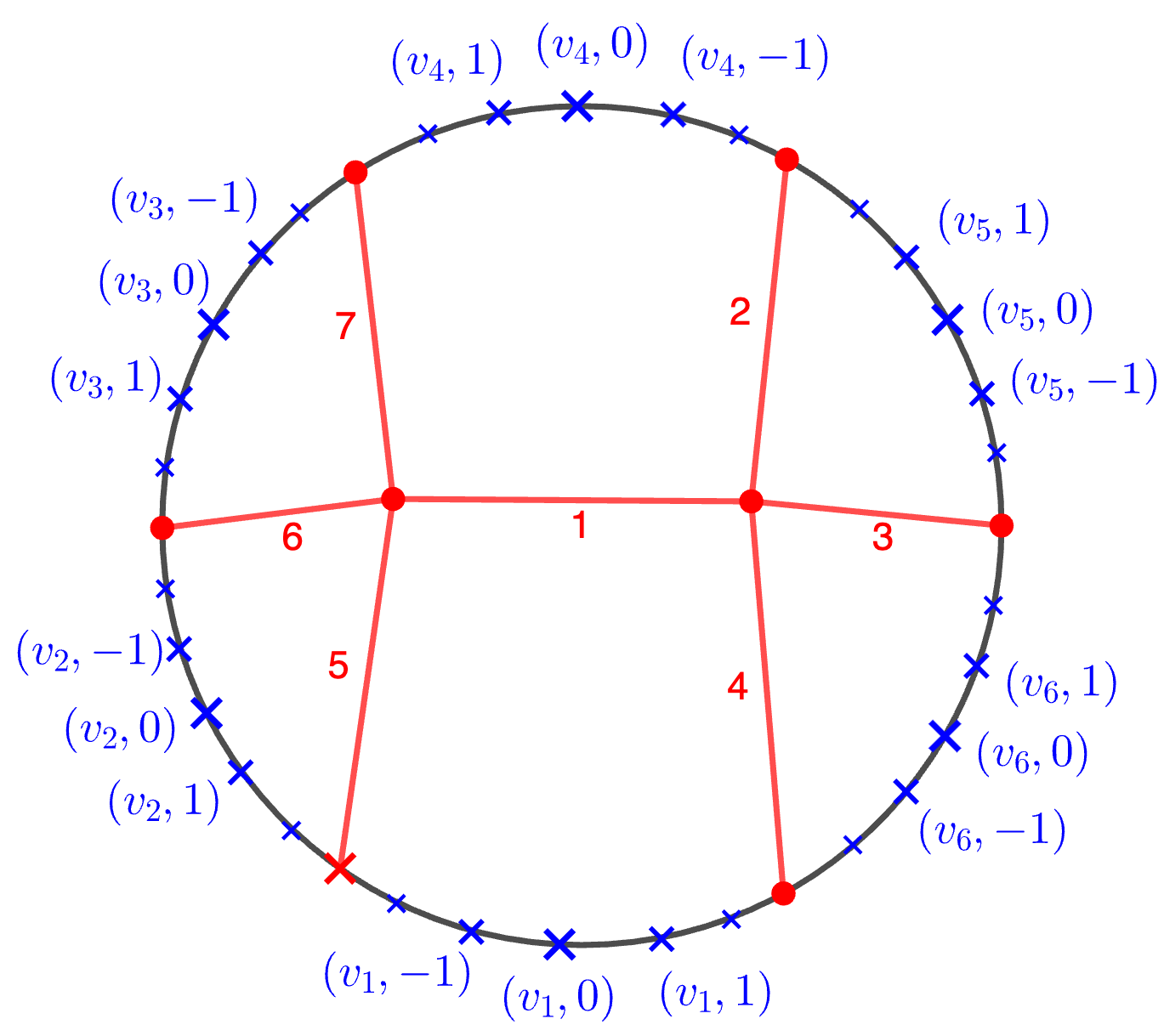}
        \caption{}
        \label{transform}
        \end{subfigure}
        \caption{}
    \end{figure}
\end{example}

\begin{definition}
A $\color{blue}\times$-arc is a smooth map $\gamma$ from the interval $[0, 1]$ to $S \setminus P$ such that its endpoints $\gamma(0)$ and $\gamma(1)$ are $\color{blue}\times$-points. The curve $\gamma$
is required not to be contractible to a $\color{blue}\times$-point.
\end{definition}

\begin{definition}
    Let $\gamma$ be a $\color{blue}\times$-arc. We assume $\gamma$ intersects the arcs of the dual $\rc$-dissection $\Delta^*$ at a finite number of points, minimally and transversally. Define $$f_{\gamma}:\gamma \cap \Delta^*\to\mathbb{Z}$$ as follows. The orientation of $\gamma$ induces a total order on the points in $\gamma\cap\Delta^*$. Let $p_0$ be the initial point in this order. Then $f(p_0):=\epsilon(s(\gamma))$. Next, if $p$ and $q$ are in~$\gamma \cap \Delta^*$ and $q$ is the successor of $p$, then $\gamma$ enters a polygon enclosed by~$\color{red}\bullet$-arcs of $\Delta^*$ via $p$ and leaves it via $q$. If the $\color{blue}\times$-points in this polygon are to the left of $\gamma$, then $f(q) := f(p) + 1$; otherwise, $f(q) := f(p) - 1$.

    A $\color{blue}\times$-arc $\gamma$ is called a \emph{slalom} if $f(q_0)=\epsilon(t(\gamma))$, where $q_0$ is the last point of~$\gamma \cap \Delta^*$.
\end{definition}

\begin{example}
    Let $\gamma_1$ be the slalom in Figure \ref{slalom}. Then the corresponding function $f_{\gamma_1}$ is given as $$f_{\gamma_1}(A)=0,\ f_{\gamma_1}(B)=-1,\ f_{\gamma_1}(C)=0.$$
\end{example}

Let $\gamma$ be a slalom connecting $T_{(i,v_1)}$ to $T_{(j,v_2)}$. We consider the homotopy class of $\gamma$ where the homotopies are allowed to move the endpoints of $\gamma$ along the boundary without crossing a $\rc$-point. Then this homotopy class contains a unique (up to homotopy) $\color{teal} \circ$-arc connecting $v_1$ and~$v_2$. We denote this $\color{teal} \circ$-arc by $\sigma_\gamma$. Moreover, the function $f_\gamma$ naturally defines a grading on the arc $\sigma_\gamma$. We will denote the graded $\color{teal} \circ$-arc $(\sigma_\gamma,f_\gamma)$ as $\Sigma(\gamma)$.

\begin{example}
    Let $\gamma_1$ be the slalom in Figure \ref{slalom}. Then using the procedure described above, $\gamma_1$ is transformed into the graded arc $\Sigma(\gamma_1)$ shown in Figure \ref{grad_arc}. 
    \begin{figure}[H]
        \centering
        \begin{subfigure}{.5\textwidth}
        \centering
        \includegraphics[scale=0.16]{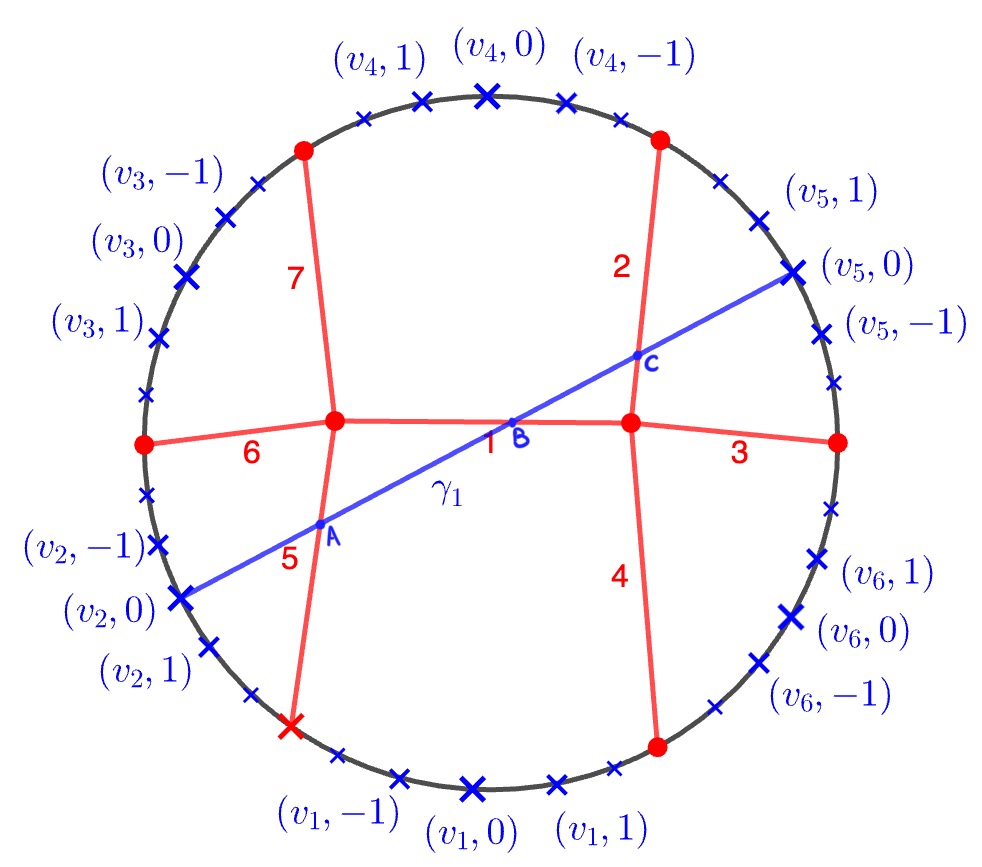}
        \caption{}
        \label{slalom}
        \end{subfigure}%
        \begin{subfigure}{.5\textwidth}
        \centering
        \includegraphics[scale=0.17]{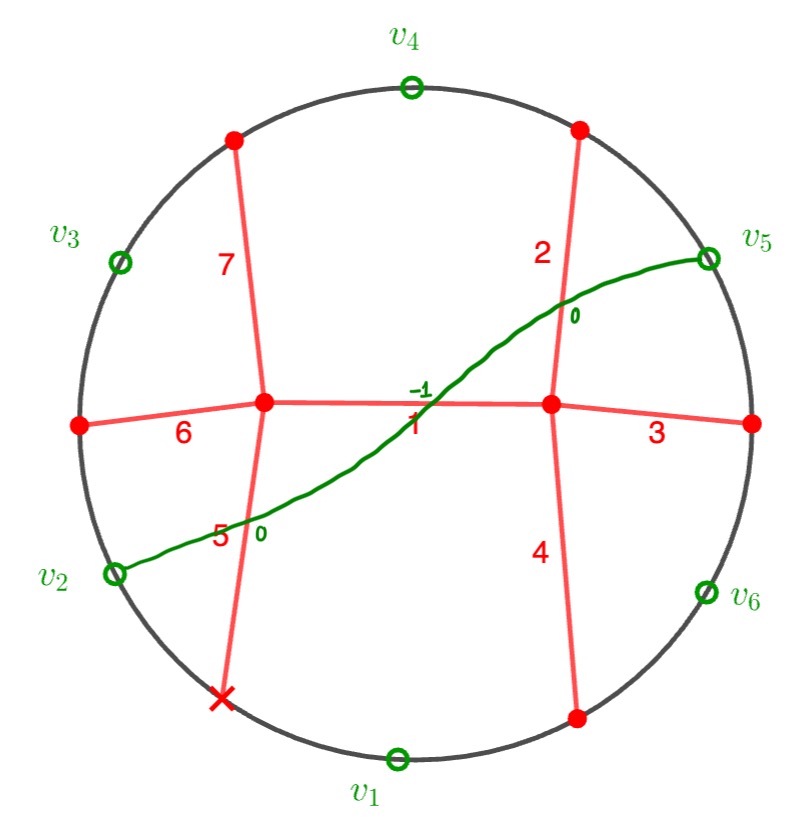}
        \caption{}
        \label{grad_arc}
        \end{subfigure}
        \caption{}
    \end{figure}
\end{example}


The map $\Sigma$ defines a bijection between the set of homotopy classes of slaloms with the set of homotopy classes of graded $\color{teal} \circ$-arcs, as defined in \cite[Definition~2.4]{APS}, which are, in turn, in bijection with certain indecomposable objects of~$K^{-,b}(\proj \La)$ called (finite) string objects. We denote by $P^\bullet_\gamma$ the object associated to $\Sigma(\gamma)$. For two graded $\color{teal} \circ$-arcs $(\mu_1,f_1)$ and $(\mu_2,f_2)$, we will use the description of $\Hom(P^\bullet_{(\mu_1,f_1)},P^\bullet_{(\mu_2,f_2)})$ in terms of the `graded intersection points' of $(\mu_1,f_1)$ and~$(\mu_2,f_2)$, as given in \cite{OPS}.

\begin{definition}
    Let $\gamma_1$ and $\gamma_2$ be two slaloms and $T$ an intersection point of $\gamma_1$ and $\gamma_2$ lying in the interior of $S$. Then $T$ is called \textbf{contractible} if (one of) the region(s) bounded by the following three segments is contractible (Figure \ref{contractible}):
    \begin{enumerate}
        \item the part of $\gamma_1$ between $T$ and the boundary.
        \item the part of $\gamma_2$ between $T$ and the boundary.
        \item the part of the boundary between the endpoints of $\gamma_1$ and $\gamma_2$.
    \end{enumerate}
    \begin{figure}[!h]
        \centering
        \includegraphics[scale=0.1]{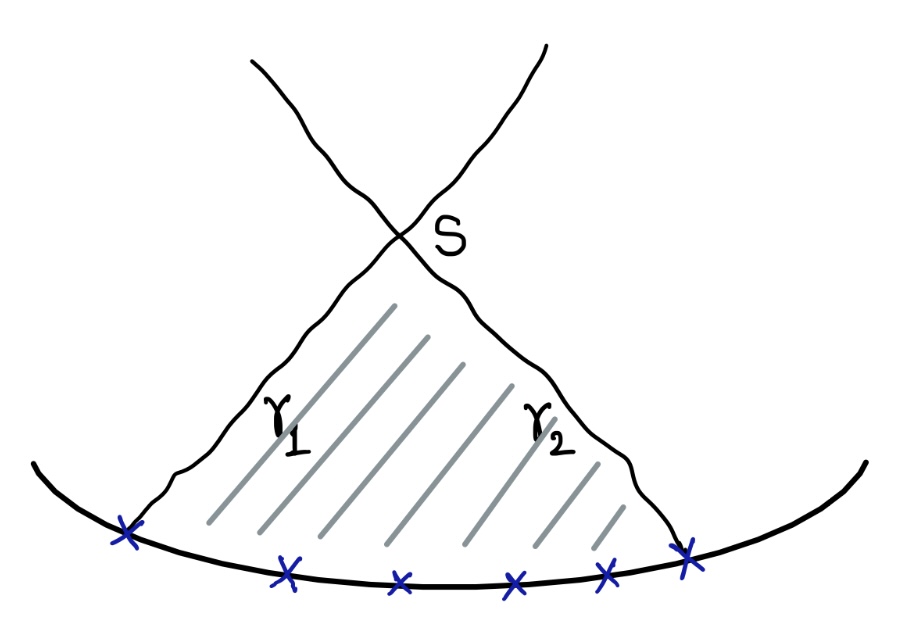}
        \caption{$S$ is a contractible intersection point of $\gamma_1$ and $\gamma_2$}
        \label{contractible}
    \end{figure}
\end{definition}

\begin{theorem}\label{main}
    Let $\gamma_1$ and $\gamma_2$ be two slaloms. Then we have the following:
    \begin{enumerate}
        \item There is a bijection between the set of interior intersection points of $\Sigma({\gamma_1})$ and $\Sigma({\gamma_2})$, and the set of non-contractible intersection points of $\gamma_1$ and~$\gamma_2$.
        \item There is a bijection between the set of boundary intersection points of~$\Sigma({\gamma_1})$ and $\Sigma({\gamma_2})$ of positive degree, and the set of contractible intersection points of $\gamma_1$ and $\gamma_2$.
        \item There is a bijection between the set of boundary intersection points of~$\Sigma({\gamma_1})$ and $\Sigma({\gamma_2})$ of degree $0$, and the boundary intersection points of $\gamma_1$ and $\gamma_2$.
    \end{enumerate}
\end{theorem}
\begin{proof}
Let $(\mu_1,f_1):=\Sigma(\gamma_1)$ and $(\mu_2, f_2):=\Sigma(\gamma_2)$. Assume that we have chosen homotopy representatives such that $\mu_1$ and $\mu_2$ intersect minimally, and so do~$\gamma_1$ and $\gamma_2$. We prove the three statements one by one.

\begin{enumerate}
    \item Let $p$ be an interior intersection point of $\Sigma(\gamma_1)$ and $\Sigma(\gamma_2)$. It is possible to choose homotopy representatives of $\gamma_1$ and $\gamma_2$ such that $p$ is also an interior intersection point of $\gamma_1$ and $\gamma_2$. Moreover, $p$ is non-contractible because, otherwise, $\mu_1$ and $\mu_2$ do not intersect minimally as we can choose homotopy representatives that do not intersect at $p$ (Figure \ref{CIP}). 
    \begin{figure}[!h]
        \centering
        \includegraphics[scale=0.15]{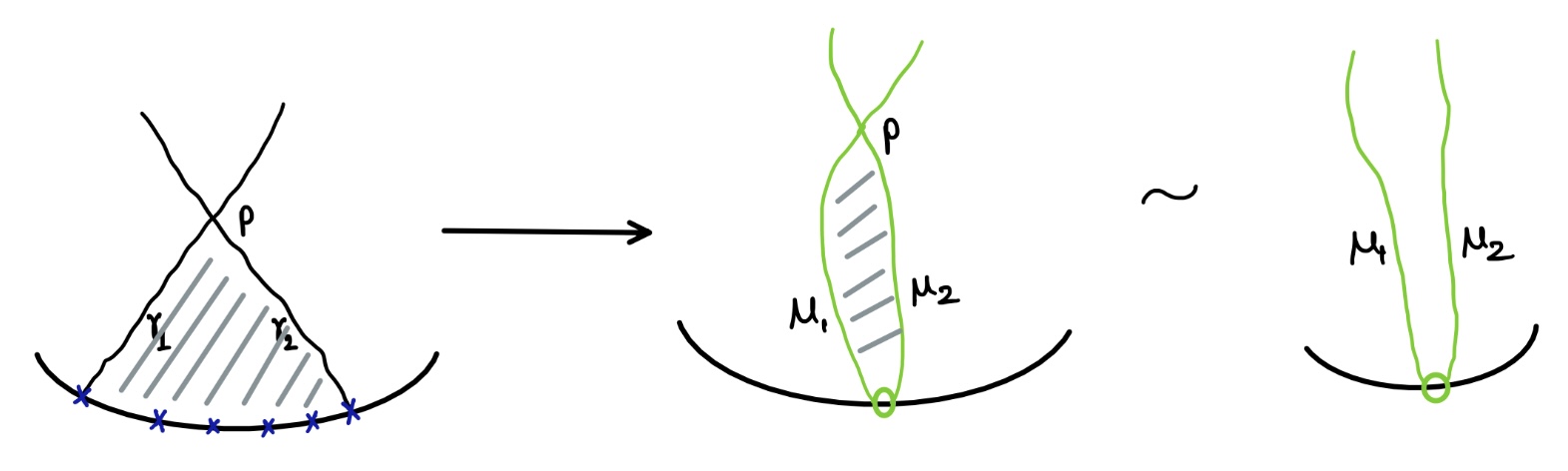}
        \caption{$p$ has to be a non-contractible intersection point of $\gamma_1$ and $\gamma_2$}
        \label{CIP}
    \end{figure}
    Conversely, if $p$ is a non-contractible intersection point of $\gamma_1$ and $\gamma_2$, then it is still an interior intersection point of $\Sigma(\gamma_1)$ and $\Sigma(\gamma_2)$ when they intersect minimally. 
    \item Suppose $p$ is a boundary intersection point of $\Sigma(\gamma_1)$ and $\Sigma(\gamma_2)$ of degree~$>0$. Then, without loss of generality, we can assume that we have the configuration around $p$ as shown in Figure \ref{pos_deg} with~$f_2(p_2)>f_1(p_1)$. 
    \begin{figure}[!h]
        \centering
        \begin{subfigure}{.5\textwidth}
        \centering
        \includegraphics[scale=0.15]{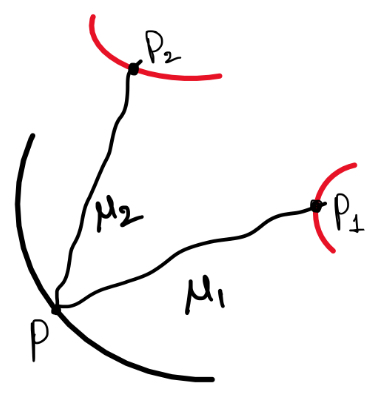}
        \caption{}
        \label{pos_deg}
        \end{subfigure}%
        \begin{subfigure}{.5\textwidth}
        \centering
        \includegraphics[scale=0.15]{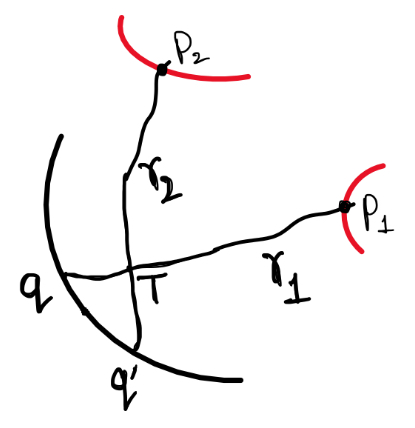}
        \caption{}
        \label{pos_deg2}
        \end{subfigure}
        \caption{}
    \end{figure}
    This implies that $\epsilon(q')>\epsilon(q)$, where $q$ (resp.~$q'$) is the endpoint of $\gamma_1$ (resp.~$\gamma_2$) such that~$\sigma(q)=p$ (resp.~$\sigma(q')=p$). Thus, the configuration of $\gamma_1$ and $\gamma_2$ is as shown in Figure \ref{pos_deg2}. Since $\mu_1$ and $\mu_2$ are homotopic to $\gamma_1$ and $\gamma_2$ respectively, and they intersect on the boundary, the point $T$ has to be a contractible intersection point of $\gamma_1$ and $\gamma_2$. 
    
    Conversely, let $T$ be a contractible intersection point of $\gamma_1$ and $\gamma_2$. Then, without loss of generality, we can assume that $\gamma_1$ and $\gamma_2$ are as shown in Figure \ref{pos_deg2}. In particular, the part of the boundary between $q$ and $q'$ does not contain any $\color{red}\bullet$-points, which implies that $\sigma(q)=\sigma(q')$, and hence $\Sigma({\gamma_1})$ and $\Sigma({\gamma_2})$ intersect on the boundary. Moreover, this intersection point has a positive degree because $f_2(p_2)=\epsilon(q')>\epsilon(q)=f_1(p_1)$.

    \item Suppose $p$ is a boundary intersection point of $\Sigma(\gamma_1)$ and $\Sigma(\gamma_2)$ of degree~$0$. Then we can assume that the configuration around $p$ is as shown in Figure~\ref{pos_deg} such that~$f_1(p_1)=f_2(p_2)$. This implies that $\epsilon(q)=\epsilon(q')=f_1(p_1)$, where $q$ (resp.~$q'$) is the endpoint of $\gamma_1$ (resp.~$\gamma_2$) such that $\sigma(q)=p$ (resp.~$\sigma(q')=p$). Thus, $q=q'$ is a boundary intersection point of $\gamma_1$ and~$\gamma_2$. Conversely, if $q$ is a boundary intersection point of $\gamma_1$ and~$\gamma_2$, then, by definition, $\Sigma(\gamma_1)$ and $\Sigma(\gamma_2)$ have an intersection point on the boundary of degree $0$. 
\end{enumerate}
\end{proof}
\begin{proposition}\label{posext}
    Two slaloms $\gamma_1,\gamma_2$ intersect in the interior of $S$ if and only if either $\Ext^i(P^\bullet_{\gamma_1},P^\bullet_{\gamma_2})\neq 0$ for some $i>0$ or $\Ext^i(P^\bullet_{\gamma_2},P^\bullet_{\gamma_1})\neq 0$ for some $i>0$.
\end{proposition}
\begin{proof}
    Suppose $\gamma_1$ and $\gamma_2$ intersect in the interior of $S$ at a point $T$. If $T$ is a contractible intersection point, then using Theorem \ref{main}, it corresponds to a boundary intersection point of $\Sigma(\gamma_1)$ and $\Sigma(\gamma_2)$ of positive degree. This implies that either $\Ext^i(P^\bullet_{\gamma_1},P^\bullet_{\gamma_2})\neq 0$ for some $i>0$ or $\Ext^i(P^\bullet_{\gamma_2},P^\bullet_{\gamma_1})\neq 0$ for some $i>0$. On the other hand, if $T$ is a non-contractible intersection point, then it corresponds to an interior intersection point of $\Sigma(\gamma_1)$ and $\Sigma(\gamma_2)$. Using \cite[Lemma~3.5]{APS}, we get that either $\Ext^i(P^\bullet_{\gamma_1},P^\bullet_{\gamma_2})\neq 0$ for some $i>0$ or $\Ext^i(P^\bullet_{\gamma_2},P^\bullet_{\gamma_1})\neq 0$ for some $i>0$.

    Conversely, suppose $\gamma_1$ and $\gamma_2$ do not intersect in the interior of $S$. Then either they do not intersect at all or they intersect on the boundary. In both cases, $\Sigma(\gamma_1)$ and $\Sigma(\gamma_2)$ can only have zero or negative degree boundary intersection points using Theorem \ref{main}. This implies that $\Ext^i(P^\bullet_{\gamma_1},P^\bullet_{\gamma_2})= 0$ for all~$i>0$ and $\Ext^i(P^\bullet_{\gamma_2},P^\bullet_{\gamma_1})= 0$ for all $i>0$.
\end{proof}

The above proposition helps us to give a characterization of presilting and silting objects in $K^{-,b}(\proj\La)$.
\begin{definition}
Say that a collection $\gamma_1,\gamma_2,\ldots, \gamma_m$ of slaloms is \textbf{mutually non-intersecting} if $\gamma_i$ does not intersect $\gamma_j$ in the interior of $\S$ for all $1\leq i,j\leq n$.
\end{definition}

\begin{corollary}
    There is a bijection between basic presilting objects in $K^{-,b}(\proj\La)$ and mutually non-intersecting collections of slaloms in $S_\La$.
\end{corollary} 
\begin{proof}
    Suppose $T=\oplus_{i=1}^mT_i$ is a basic presilting object in $K^{-,b}(\proj\La)$ with $T_i$ indecomposable. Then, using \cite[Lemma~3.4]{APS}, we get that $T_i$ is of the form $P^{\bullet}_{\gamma_i}$ for some slalom $\gamma_i$. Since $T$ is presilting, we get that $\Ext^j(T_{i_1},T_{i_2})= 0$ for all~$j>0$ and for all $1\leq i_1, i_2\leq m $. Using the previous proposition, this gives that~$\{\gamma_1,\cdots,\gamma_m\}$ is a mutually non-intersecting collection of slaloms in $S_\La$.
    
    Conversely, given a mutually non-intersecting collection of slaloms $\{\gamma_1,\cdots,\gamma_m\}$ in $S_\La$, the previous proposition immediately implies that $T=\oplus_{i=1}^mP^{\bullet}_{\gamma_i}$ is a basic presilting object in $K^{-,b}(\proj\La)$. 
\end{proof}

Using \cite[Proposition~5.7]{APS}, we know that for a gentle algebra $\La$, a basic presilting object $X$ is silting if and only if it has $n=|\La|$ many indecomposable summands. This gives us the following corollary.

\begin{corollary}
    There is a bijection between basic silting objects in $K^{-,b}(\proj\La)$ and mutually non-intersecting collections of $n$ slaloms in $S_\La$.
\end{corollary} 

We can also restrict the above model to obtain a characterization of $d$-term presilting objects of $\La$. To do this, we keep only the $\color{blue}\times$-points labelled from $0$ to~$-d+1$ between each pair of neighbouring $\rc$-points. A slalom in the restricted model is a slalom $\gamma$ in the original model for which $f_{\gamma}(\gamma\cap\Delta^*)\subseteq [-d+1,0]$, as this ensures that $P^{\bullet}_\gamma$ is concentrated in $[-d+1,0]$. We will denote this restricted model by $S_\La^d$. It is easy to see that Theorem \ref{main} and Proposition \ref{posext} still hold in this restricted model, which gives us the following corollary.

\begin{corollary}\label{dsilt}
    There is a bijection between basic $d$-term silting objects of $\La$ and mutually non-intersecting collections of $n$ slaloms in $S^d_\La$.
\end{corollary}

\begin{remark}
Note that in the full model $S_\La$, for every finite arc starting at a $\color{blue}\times$-point, there exists a unique choice for the other endpoint which makes it a slalom. However, this is not the case in the restricted model $S^d_\La$, because of the additional condition that ${\gamma}(\gamma\cap\Delta^*)\subseteq [-d+1,0]$.
\end{remark}
\begin{remark}\label{d=2}
    In \cite{PPP19}, the authors introduced a model for basic support $\tau$-tilting modules by considering blossom points on the boundary of a marked surface and slaloms connecting these points. Viewing these blossom points as $\color{blue}\times$-points, one precisely recovers the model $S^2_\La$ described above, which is consistent with the fact that basic support $\tau$-tilting modules are in bijection with basic $2$-term silting complexes \cite{AIR}. 
\end{remark}

\section{Counting $d$-term silting objects in linearly oriented $A_n$}\label{count}
Let $kA_n$ be the path algebra of the linearly oriented quiver of type $A_n$, that is,~$kA_n=k(1\to2\to\cdots\to n)$. In this section, we will give a recursive formula for the number of $d$-term silting objects in $kA_n$ using the above corollary. This will recover the result of \cite{STW} that these are counted by the Pfaff-Fuss-Catalan numbers. 

The marked surface $(S,M,P)$ obtained from Theorem \ref{MStoGA} for $\La=kA_n$ is shown in Figure \ref{S_n_green}. By replacing each $\color{teal}\circ$-point in this figure with $d$ $\color{blue}\times$-points, we obtain the surface $S^d_{kA_n}$ defined above (Fig \ref{S^d_n}). 

\begin{figure}[!h]
        \centering
        \begin{subfigure}{.5\textwidth}
        \centering
        \includegraphics[scale=0.22]{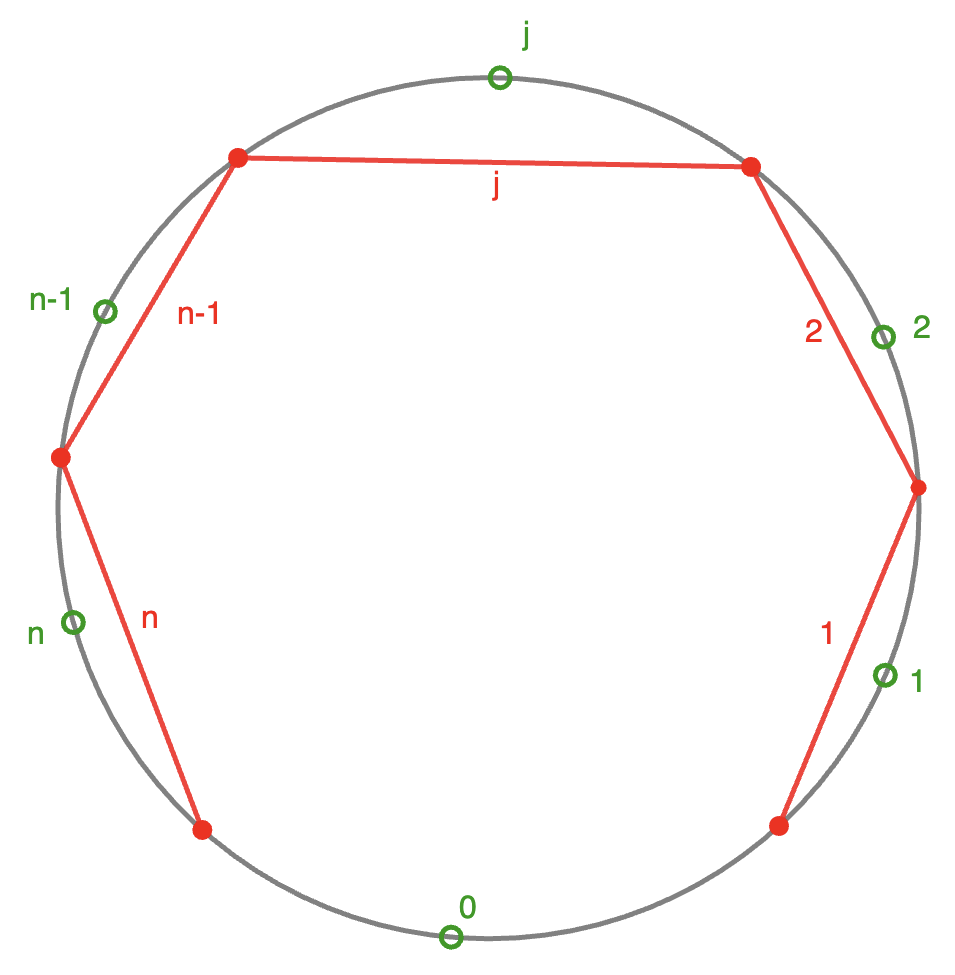}
        \caption{}
        \label{S_n_green}
        \end{subfigure}%
        \begin{subfigure}{.5\textwidth}
        \centering
        \includegraphics[scale=0.26]{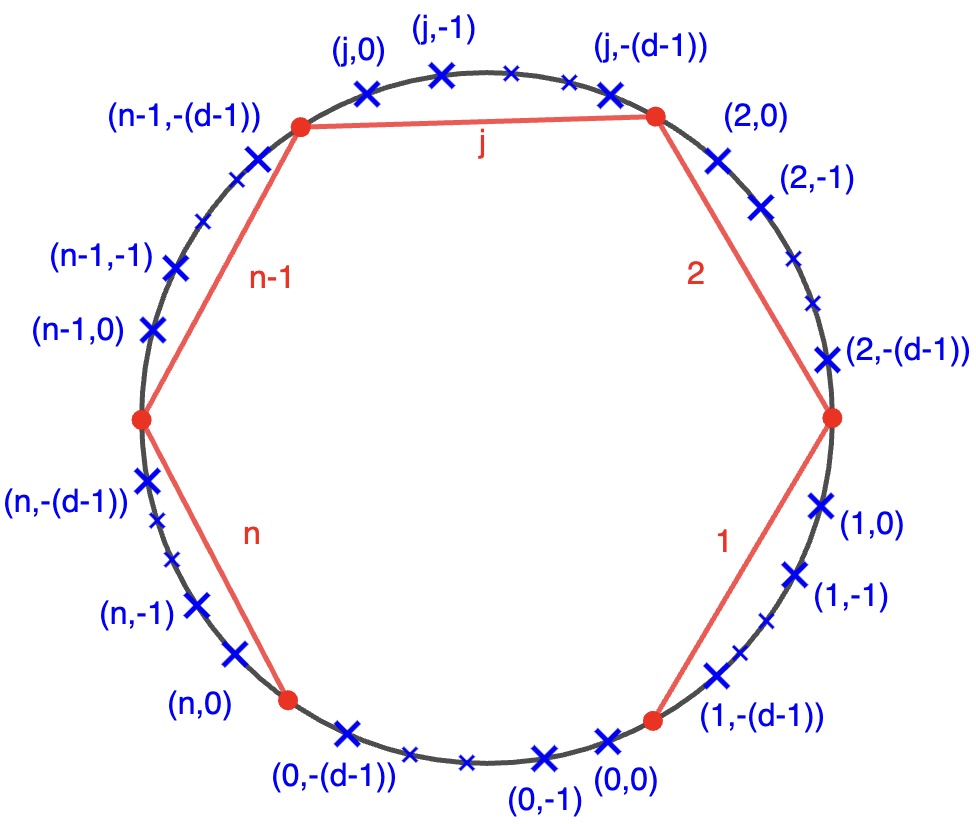}
        \caption{}
        \label{S^d_n}
        \end{subfigure}
        \caption{}
\end{figure}

Let $B^d_n$ denote the number of basic silting objects in $kA_n$. Using Corollary~\ref{dsilt}, we get that $B^d_n$ is also the number of collections of $n$ mutually non-intersecting slaloms in Figure \ref{S^d_n}. To calculate $B^d_n$, we first count the number of such collections containing some fixed slalom $\gamma$. This is done by cutting the disc along this slalom, and relabelling (one of) the parts thus obtained to get $S^d_{kA_m}$ for some $m<n$. This allows us to build a recursive formula for $B^d_n$. We explain the detailed process below.

Let $\Gamma$ be a collection of $n$ mutually non-intersecting slaloms. Such a collection will be maximal with respect to the property of mutual non-intersection as the number of indecomposable summands of any presilting object is less than or equal to $n$ \cite{XY}. Our first claim is that in such a collection of $n$ slaloms, at least one of~$(0,0), (0,-1), (0,-2), \cdots (0,-d+1)$ has to be an endpoint of some slalom. This is because otherwise, the slalom in Figure \ref{step1} will contradict the maximality of the collection. 

\begin{figure}
        \centering
        \begin{subfigure}{.5\textwidth}
        \centering
        \includegraphics[scale=0.25]{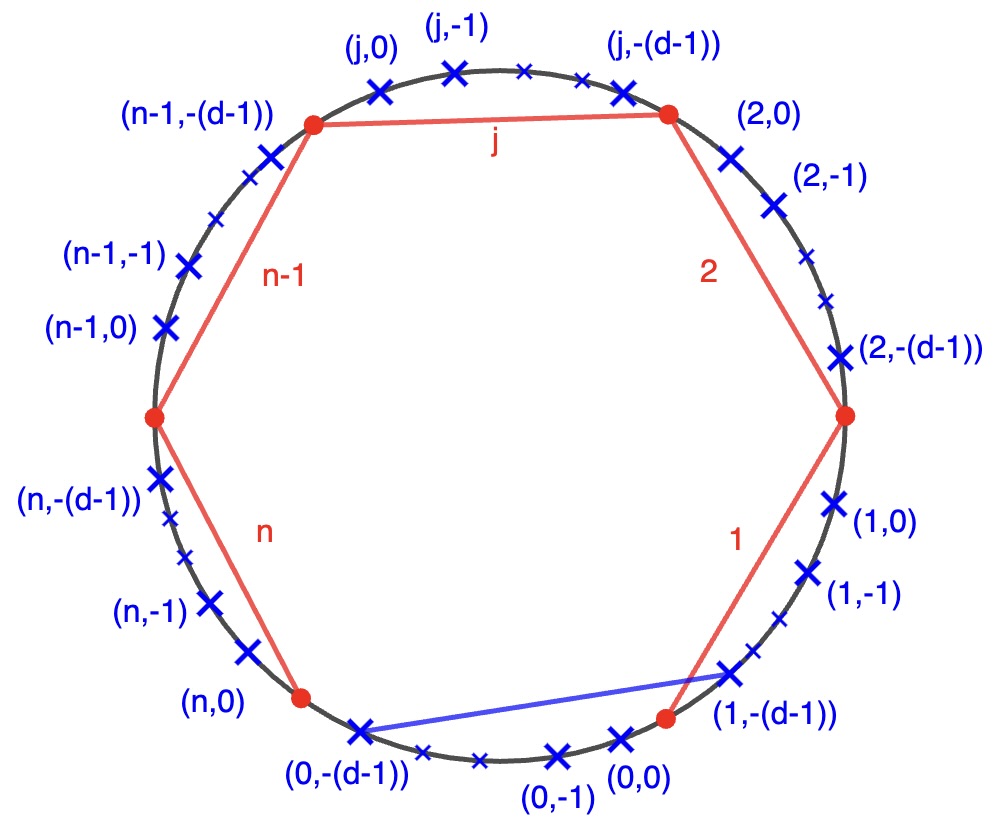}
        \caption{}
        \label{step1}
        \end{subfigure}%
        \begin{subfigure}{.5\textwidth}
        \centering
        \includegraphics[scale=0.25]{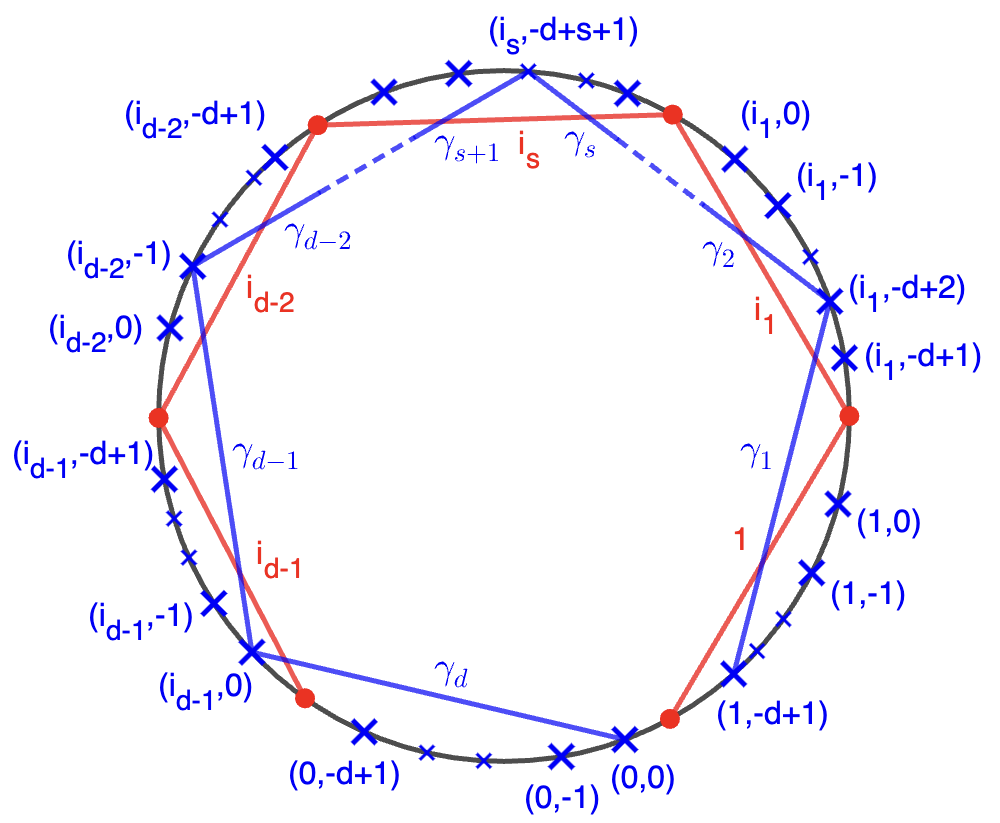}
        \caption{}
        \label{step2}
        \end{subfigure}
        \caption{}
\end{figure}

\begin{lemma}\label{lemrelabel1}
    Let $\gamma$ be a slalom in $S^d_n$ as shown in Figure \ref{relabel1}. Then a $\color{blue}\times$-arc lying in the disc $D'$ is a slalom in $S^d_n$ if and only if it is a slalom in $S^d_{i_2-i_1-1}$ obtained by relabelling the points of $D'$ as shown in Figure \ref{relabel2}.
\end{lemma}
\begin{figure}
        \centering
        \begin{subfigure}{.5\textwidth}
        \centering
        \includegraphics[scale=0.15]{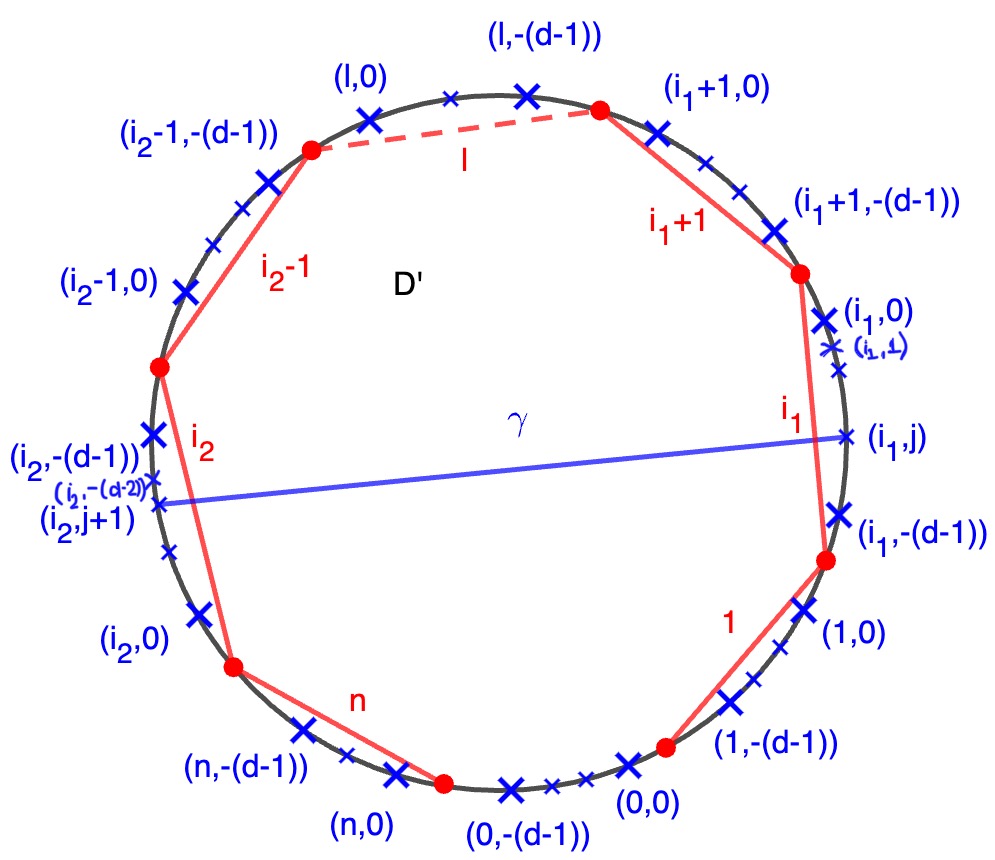}
        \caption{}
        \label{relabel1}
        \end{subfigure}%
        \begin{subfigure}{.5\textwidth}
        \centering
        \includegraphics[scale=0.15]{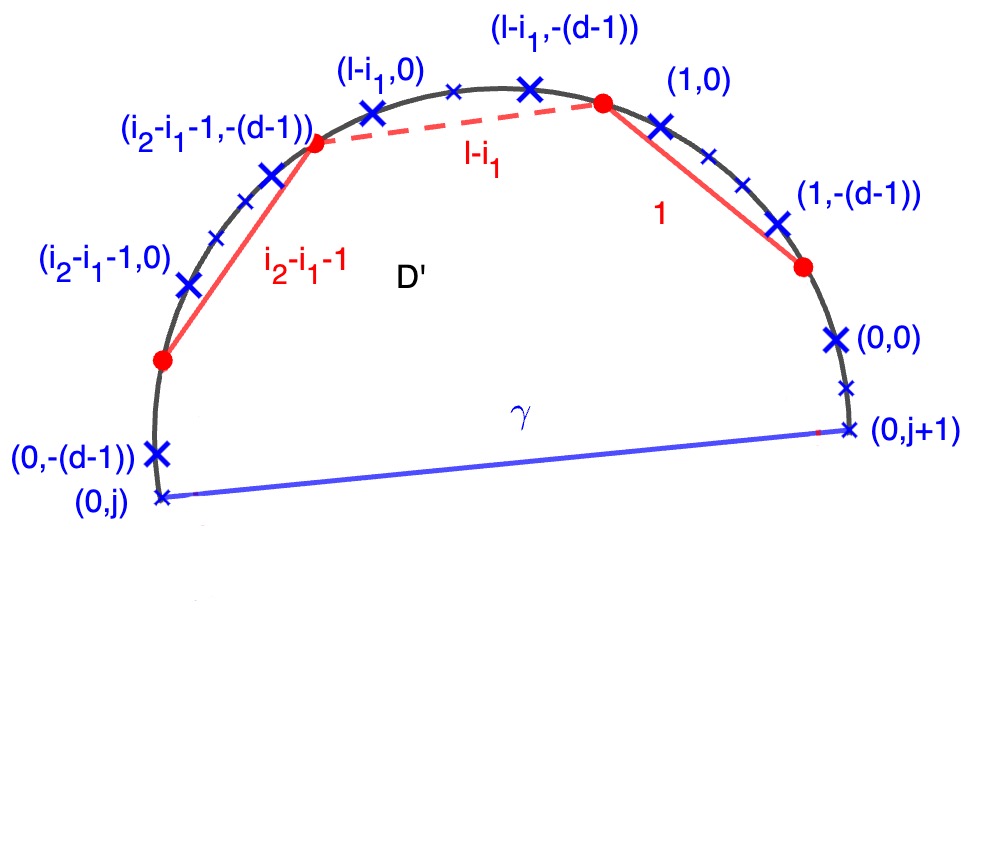}
        \caption{}
        \label{relabel2}
        \end{subfigure}
        \caption{}
\end{figure}
\begin{proof}
    Since there are no slaloms in $S^d_n$ starting at $(i_1,0)$ and lying in $D'$, we can remove it. Moreover, since the only slaloms in $D'$ starting at $(i_1,m)$ for some~$j\leq~m\leq -1$ end at $(l, m+1)$ for some~$i_1+1\leq l\leq i_2-1$, removing the $\rc$-arc $i_1$ and relabelling as in Fig \ref{relabel2} gives a bijection between the set of slaloms starting at these points. Similarly, since the only slaloms in $D'$ ending at~$(i_2,m)$ for some $-d+1\leq m \leq j+1$ start at $(l,m-1)$ for some $i_1+1\leq l\leq i_2-1$, removing the $\rc$-arc $i_2$, the point $(i_2,-d+1)$, and relabelling as in Fig \ref{relabel2} gives a bijection between the set of slaloms ending at these points. Combining this with the map that sends the slalom connecting $(s_1,t)$ to $(s_2,t+1)$ for some $i_1<s_1<s_2<i_2$ in~$D'$ to the slalom connecting $(s_1-i_1,t)$ to $(s_2-i_1,t+1)$ in the relabelled figure, we get the required bijection. 
\end{proof}
\begin{lemma}\label{lemrelabel2}
    Let $\gamma$ be a slalom in $S^d_n$ as shown in Figure \ref{relabel1.1}. Then a $\color{blue}\times$-arc lying in the disc $D'$ is a slalom in $S^d_n$ if and only if it is a slalom in $S^d_{n-i}$ obtained by relabelling the points of $D'$ as shown in Figure \ref{relabel1.2}.
\end{lemma}
\begin{figure}
        \centering
        \begin{subfigure}{.5\textwidth}
        \centering
        \includegraphics[scale=0.15]{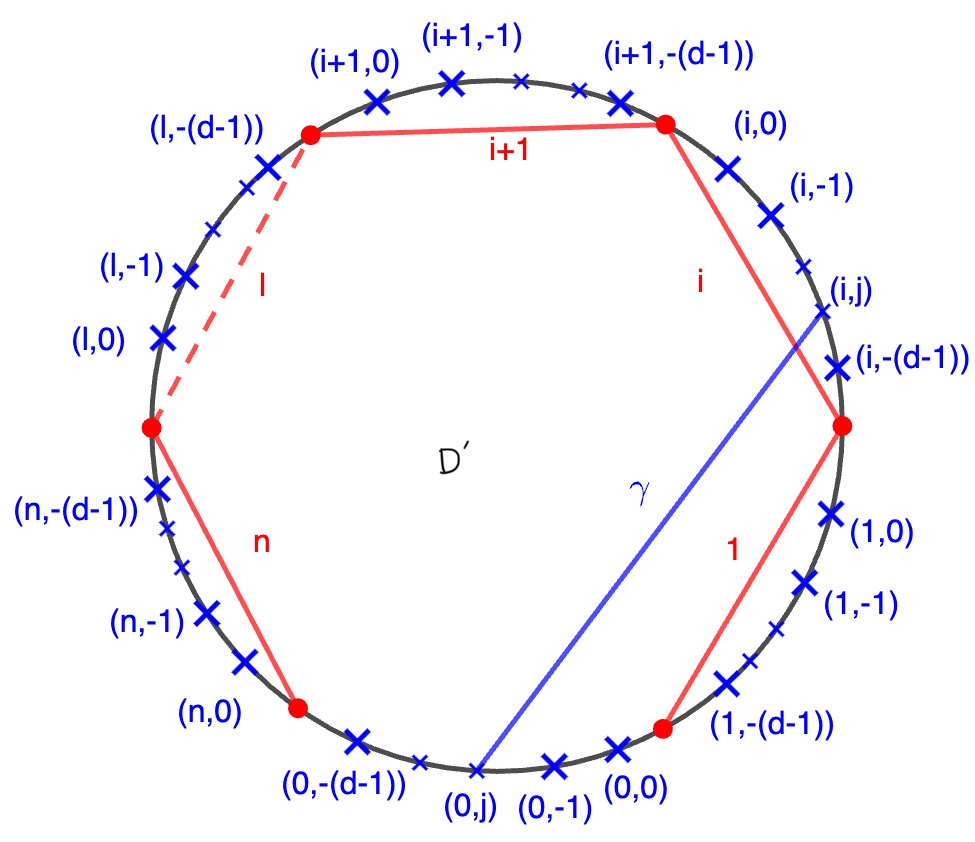}
        \caption{}
        \label{relabel1.1}
        \end{subfigure}%
        \begin{subfigure}{.5\textwidth}
        \centering
        \includegraphics[scale=0.15]{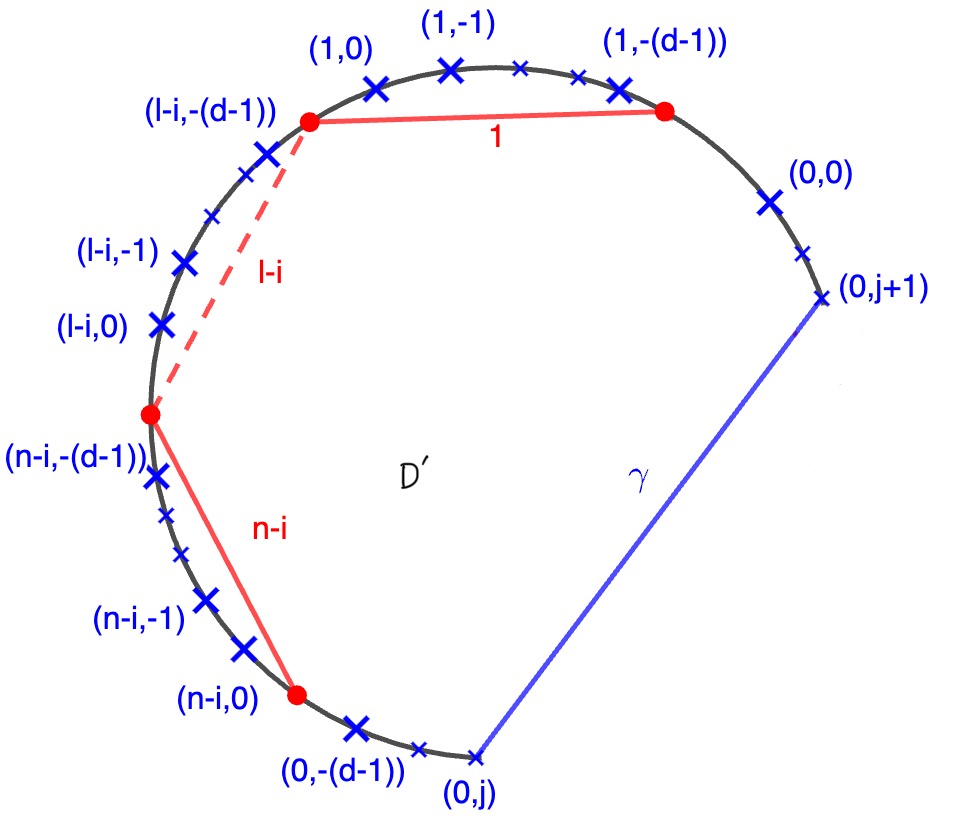}
        \caption{}
        \label{relabel1.2}
        \end{subfigure}
        \caption{}
\end{figure}
\begin{proof}
    Since there are no slaloms in $S^d_n$ starting at $(i,0)$ and lying in~$D'$ (other than possibly $\gamma$ when $j=0$), we can remove it. Moreover, since the only slaloms in $D'$ starting at $(i,m)$ for some $j\leq m \leq -1$ end at $(l, m+1)$ for some $i+1\leq l\leq n$, removing the $\rc$-arc $i$ and relabelling as in Fig \ref{relabel1.2} gives a bijection between the set of slaloms starting at these points. Combining this with the map that sends the slalom connecting $(s_1,t)$ to $(s_2,t+1)$ for some $i<s_1<s_2\leq n$ in $D'$ to the slalom connecting $(s_1-i,t)$ to $(s_2-i,t+1)$ in the relabelled figure, and the slalom connecting $(s,t)$ to $(0,t)$, for some $i<s_1\leq n$ and $-d+1\leq t\leq j$, in $D'$ to the slalom connecting $(s-i,t)$ to $(0,t)$ in the relabelled figure, we get the required bijection. 
\end{proof}
\begin{corollary}
    The number of mutually non-intersecting collections of $i_2-i_1-1$ slaloms in $S^d_n$ lying in $D'$ is $B^d_{i_2-i_1-1}$.
\end{corollary}
\begin{proof}
    It is easy to see that a collection of $i_2-i_1-1$ slaloms in $S^d_n$ lying in $D'$ is mutually non-intersecting if and only if it is mutually non-intersecting in $S^d_{i_2-i_1-1}$ under the above bijection. Using Corollary \ref{dsilt}, the number of such collections is equal to $B^d_{i_2-i_1-1}$.
\end{proof}

We can now divide the problem into two cases. 
\begin{enumerate}
    \item $(0,0)$ is the endpoint of a slalom: Let $\gamma_d$ be the rightmost such slalom. Since this is the rightmost slalom connected to $(0,0)$, we need to block the slaloms connecting $(0,0)$ to some $(j,0)$ for $j<i_{d-1}$. This gives that at least one of $(1,-1), \cdots (1,-d+1)$ is the endpoint of some slalom in $\Gamma$. Suppose $(1,-d+1)$ is the endpoint of some slalom in $\Gamma$ and let $\gamma_1$ be the leftmost such slalom. To block the slalom connecting $(0,0)$ to $(i_1,0)$, we need that at least one of $(i_1,0), \cdots (i_1,-d+2)$ is the endpoint of some slalom in $\Gamma$. Repeating this argument with the assumption that at each step the point with the lowest index is the endpoint of some slalom in $\Gamma$, we get a collection of $d$ slaloms in $\Gamma$ arranged as shown in Figure \ref{step2}.
    \begin{figure}[!h]
        \centering
        \includegraphics[scale=0.15]{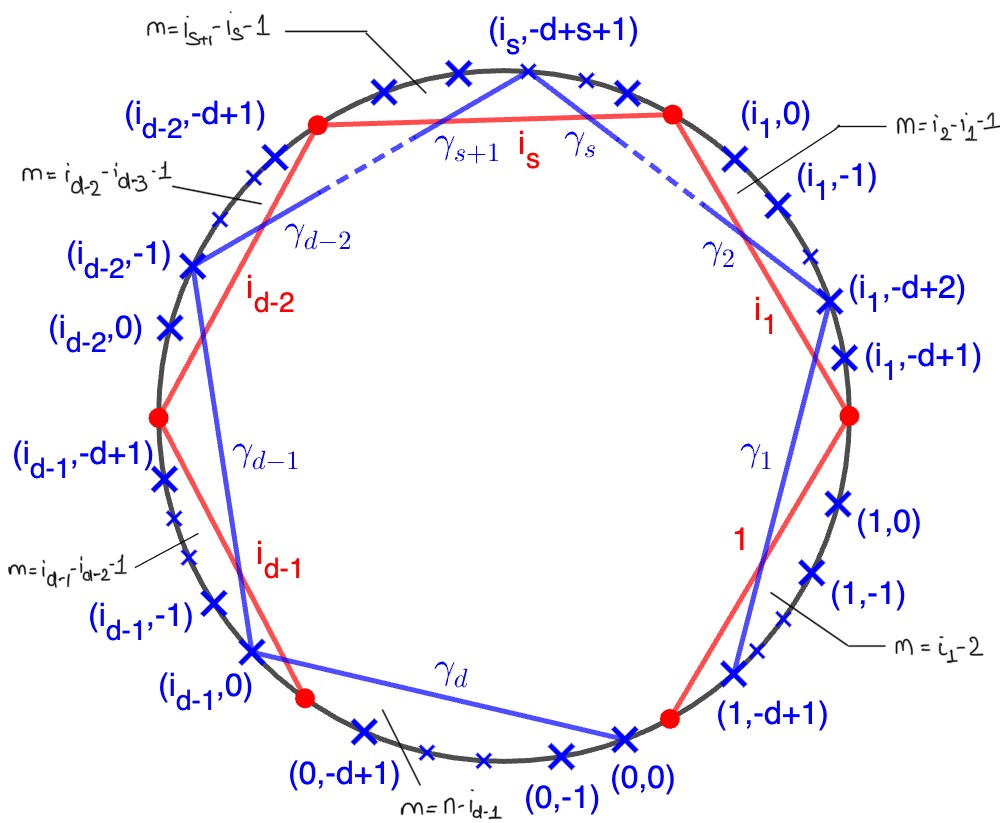}
        \caption{The bigger disc can be divided into $d$ smaller pieces}
        \label{mvalue}
    \end{figure}
    Using Lemmas \ref{lemrelabel1} and \ref{lemrelabel2}, each of the parts of the discs cut by these slaloms can be relabelled to obtain $S^d_m$ for some $m<n$. The exact value of $m$ for each of the parts is shown in Figure \ref{mvalue}. Thus the number of collections of non-mutually intersecting $n$ slaloms containing the above collection of $d$ slaloms is given by the sum of the product of the number of collections of non-mutually intersecting $a_1,\cdots, a_d$ slaloms in $S^d_{kA_{i_1-2}}, S^d_{kA_{i_2-i_1-1}}, \cdots, S^d_{kA_{i_{d-1}-i_{d-2}-1}}, S^d_{kA_{n-i_{d-1}}}$ respectively. Now, since we know that $a_1\leq i_1-2$, $a_l\leq i_l-i_{l-1}-1$ for $2\leq l\leq d-1$, $a_d\leq n-i_{d-1}$, and $\Sigma_{m=1}^da_d=n-d$, we get that $a_1=i_1-2$, $a_l= i_l-i_{l-1}-1$ for $2\leq l\leq d-1$, and $a_d= n-i_{d-1}$. This means that the above number is equal to the product $B^d_{i_1-2}B^d_{i_2-i_1-1}\cdots B^d_{i_{d-1}-i_{d-2}-1}B^d_{n-i_{d-1}}$. Taking into account all possible values of $i_1, \cdots i_{d-1}$, we get the sum $$\Sigma_{i_{d-1}=2}^n\Sigma_{i_{d-2}=2}^{i_{d-1}-1}\cdots\Sigma_{i_1=2}^{i_2-1}B^d_{i_1-2}B^d_{i_2-i_1-1}\cdots B^d_{i_{d-1}-i_{d-2}-1}B^d_{n-i_{d-1}}.$$

    Now in the penultimate step of the above argument, when one of $(i_{d-2},0)$ or $(i_{d-2},-1)$ has to be connected to something in $\Gamma$, suppose $(i_{d-2},-1)$ is not connected to anything. Then $(i_{d-2},0)$ has to be connected to something, which is only possible if $i_{d-2}=i_{d-1}$ (Figure \ref{step3}). 
    \begin{figure}
        \centering
        \begin{subfigure}{.5\textwidth}
        \centering
        \includegraphics[scale=0.25]{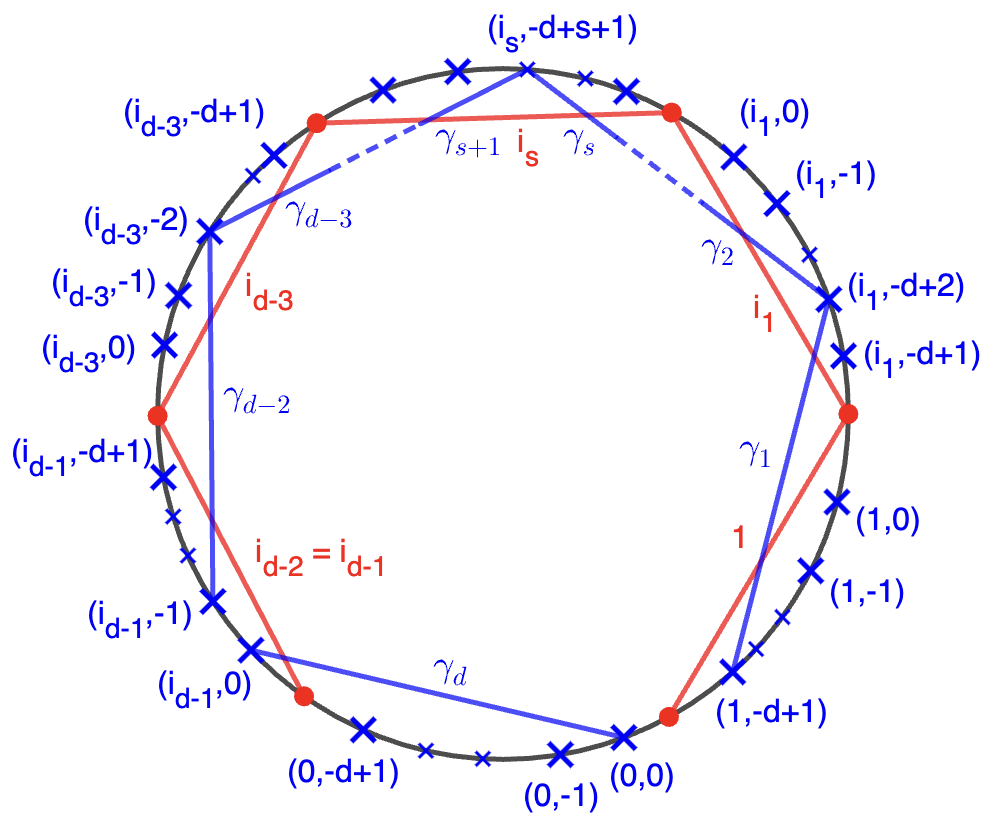}
        \caption{}
        \label{step3}
        \end{subfigure}%
        \begin{subfigure}{.5\textwidth}
        \centering
        \includegraphics[scale=0.26]{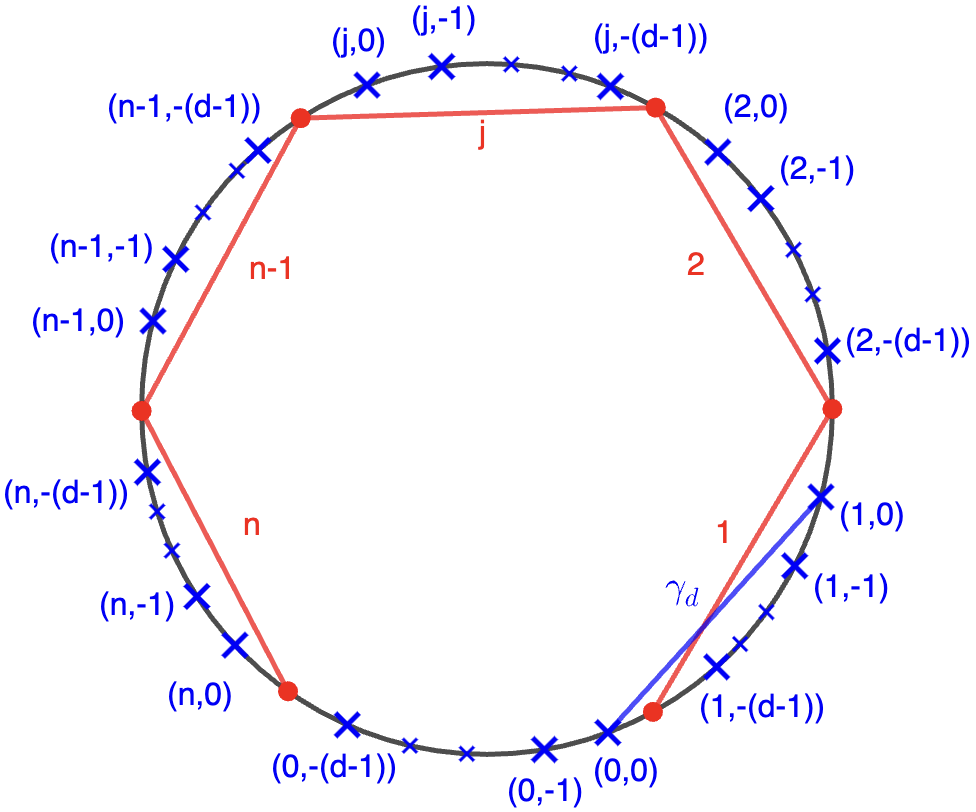}
        \caption{}
        \label{step4}
        \end{subfigure}
        \caption{}
\end{figure}
    Again, in this case, the parts of the disc cut by the $d-1$ slaloms can be relabelled by Lemmas \ref{lemrelabel1} and \ref{lemrelabel2} to get $S^d_{m}$ for some $m<n$. And we will get that the number of collections $\Gamma$ containing this set of $d-1$ slaloms is given by the product of $d-1$ terms $B^d_{i_1-2}, B^d_{i_2-i_1-1},\cdots ,B^d_{i-i_{d-3}-1},B^d_{n-i}$. Repeating this argument at each step, we get that the total number of collections $\Gamma$ in which $(0,0)$ is connected to something is given by $$\Sigma_{k=0}^{d-1}\binom{d-1}{k}\Sigma_{i_k=2}^n\Sigma_{i_{k-1}=2}^{i_k-1}\cdots\Sigma_{i_1=2}^{i_2-1}B^d_{i_1-2}B^d_{i_2-i_1-1}\cdots B^d_{i_k-i_{k-1}-1}B^d_{n-i_k},$$ where the sum over $0$ elements is taken to be $B^d_{n-1}$, corresponding to the boundary case in Figure \ref{step4}.

    \item $(0,0)$ is not the endpoint of a slalom: First suppose that $(0,-1)$ is connected to something. It can be easily seen that the above cutting and relabelling procedure works in this case as well, except that the maximum number of parts in which the disc is cut is now $d-1$. This gives that the number of collections $\Gamma$ in this case is $$\Sigma_{k=0}^{d-2}\prescript{d-2}{}{C}_k\Sigma_{i_k=2}^n\Sigma_{i_{k-1}=2}^{i_k-1}\cdots\Sigma_{i_1=2}^{i_2-1}B^d_{i_1-2}B^d_{i_2-i_1-1}\cdots B^d_{i_k-i_{k-1}-1}B^d_{n-i_k}.$$ In the general case, let $m$ be the smallest integer such that $(0,-m)$ is connected to something. Then the number of collections is given by $$\Sigma_{k=0}^{d-1-m}\binom{d-1-m}{k}\Sigma_{i_k=2}^n\Sigma_{i_{k-1}=2}^{i_k-1}\cdots\Sigma_{i_1=2}^{i_2-1}B^d_{i_1-2}B^d_{i_2-i_1-1}\cdots B^d_{i_k-i_{k-1}-1}B^d_{n-i_k}.$$ Thus, we get that the total number of $d$-term silting objects in $kA_n$ is given by the recursive formula $$B^d_n=\Sigma_{k=0}^{d-1}\Sigma_{m=k}^{d-1}\binom{m}{k}\Sigma_{i_k=2}^n\Sigma_{i_{k-1}=2}^{i_k-1}\cdots\Sigma_{i_1=2}^{i_2-1}B^d_{i_1-2}B^d_{i_2-i_1-1}\cdots B^d_{i_k-i_{k-1}-1}B^d_{n-i_k},$$ where $B_0^d=1$, and $B_1^d=d$.
    
\end{enumerate}
Since the above recursive formula also counts the number of complete ordered $d$-ary trees with $n+1$ internal nodes, which are known to be counted by the Pfaff-Fuss-Catalan numbers \cite{HP}, we recover the result that the number of $d$-term silting objects in $kA_n$ is given by the Pfaff-Fuss-Catalan number $C^d_{n+1}$.

\printbibliography
\Address
\end{document}